\documentclass[a4paper,11pt,leqno,final]{article}

\usepackage[T1]{fontenc}
\usepackage[utf8]{inputenc}

\usepackage{amsmath,amssymb,amsthm}
\usepackage{dsfont}
\usepackage{verbatim}

\usepackage[a4paper,margin=40mm]{geometry}

\usepackage{fancyhdr}

\usepackage{color}
\newcommand{\rene}{\color{red}}

\newcommand{\normal}{\color{black}}
\newcommand{\Ee}{\mathds{E}}
\newcommand{\nat}{\mathds{N}}
\newcommand{\real}{\mathds{R}}
\newcommand{\rd}{{\mathds{R}^d}}

\newcommand{\I}{\mathds{1}}
\newcommand\supp{\mathop{\operatorname{supp}}}

\newcommand{\dup}[1]{\mathrm{d}#1}

\newcommand{\form}{\mathcal{E}}
\newcommand{\dom}{\mathcal{F}}
\newcommand{\loc}{\mathrm{loc}}
\newcommand{\lip}{\mathrm{lip}}
\newcommand{\sbullet}{\text{\tiny\textbullet}}

\newcommand{\scalp}[2]{\langle #1,#2\rangle}
\newcommand{\entier}[1]{\lfloor#1\rfloor}

\numberwithin{equation}{section}
\swapnumbers

\newtheorem{theorem}{\uppercase{Theorem}}[section]
\newtheorem{lemma}[theorem]{\uppercase{Lemma}}
\newtheorem{corollary}[theorem]{\uppercase{Corollary}}

\theoremstyle{definition}

\newtheorem{remark}[theorem]{\textit{Remark}}
\newtheorem{example}[theorem]{\textit{Example}}

\title{\textbf{\uppercase{Homogenization of Symmetric L\'evy Processes on $\rd$}}}
\author{\uppercase{Ren\'e L. Schilling} \and \uppercase{Toshihiro Uemura}}

\date{\footnotesize Institut f\"{u}r Mathematische Stochastik, Fakult\"at Mathematik,
Technische Universit\"at Dresden\\ 01062 Dresden, Germany\\
Department of Mathematics, Faculty of Engineering Science, Kansai University\\ Suita-shi, Osaka 564-8680, Japan}

\date{Dedicated to the memory of Professor Nicu Boboc\\{\small To appear in \itshape Revue Roumaine des Math\'ematiques Pures et Appliqu\'ees}
}

\begin{document}

\maketitle

\begin{abstract}\noindent
    In this short note we study homogenization of symmetric $d$-dimensional L\'evy processes. Homogenization of one-dimensional pure jump Markov processes has been investigated by Tanaka \emph{et al.}\ in \cite{HIT77}; their motivation was the work by Benssousan \emph{et al.}\ \cite{BLP75} on the homogenization of diffusion processes in $\rd$, see also \cite{BLP78} and \cite{T92}. We investigate a similar problem for a class of symmetric pure-jump L\'evy processes on $\rd$ and we identify -- using Mosco convergence -- the limit process.


\end{abstract}

A symmetric L\'evy process $(X_t)_{t\geq 0}$ is a stochastic process in $\rd$ with stationary and independent increments, c\`adl\`ag paths and symmetric laws $X_t\sim -X_t$. We can characterize the (finite-dimensional distributions of the) process by its characteristic function $\Ee e^{i\scalp{\xi}{X_t}}$, $\xi\in\rd$, $t>0$, which is of the form $\exp(-t\psi(\xi))$; due to the symmetry of $X_t$, the characteristic exponent $\psi$ is real-valued. It is given by the L\'evy--Khintchine formula
\begin{equation}\label{lkf}
    \psi(\xi) = \frac 12\scalp{\xi}{\Sigma\xi} + \int_{h\neq 0}\big(1-\cos\scalp{\xi}{h}\big)\,\nu(\dup{h}),\quad\xi\in\real^d.
\end{equation}
$\Sigma\in\real^{d\times d}$ is the positive semidefinite \emph{diffusion matrix} and $\nu(\dup{h})$ is the \emph{L\'evy measure}, that is a Radon measure on $\rd\setminus\{0\}$ such that $\int_{h\neq 0} \big(1\wedge |h|^2\big)\,\nu(\dup{h})$ is finite. It is clear from \eqref{lkf} that we have $\nu(\dup{h})=\nu(-\dup{h})$.
Throughout this paper we assume that $\Sigma\equiv 0$ and that $\nu(\dup{h})$ has a (necessarily symmetric) locally bounded density on $\rd\setminus\{0\}$ w.r.t.\ Lebesgue measure; in abuse of notation we write $\nu(\dup{h}) = \nu(h)\,\dup{h}$.

Let $Q = (0,1)^d$ be the open unit cube in $\rd$ and $a:\rd\to\real$ a function in $L^p_{\loc}(\rd)$ for some $1<p\leq \infty$.
We assume that $a(x)=a(-x)$ for $x\in \real^d$ and $a$ is \emph{$Q$-periodic} in the sense that
\begin{equation}\label{periodic-1}
    a(h + k e_i) = a(h) > 0
    \ \  \text{for all\ $k \in \mathds{Z},\; i=1,2,\ldots,d$ \ and a.a.\ } h\in Q;
\end{equation}
as usual, $e_i$ denotes the $i$th unit vector of $\rd$. By $\overline{a}$ we denote the \emph{mean value} of $a$,
\begin{gather}\label{periodic-2}
    \overline{a} := \int_Q a(h)\,\dup{h}.
\end{gather}
We assume that $a_{\delta}(h) :=a\left(\delta^{-1} h\right)$ satisfies
\begin{gather}
\label{periodic-3}
    \int_{h\neq 0} \big(1 \wedge |h|^2\big) a_{\delta}(h) \nu(h)\,\dup{h} <\infty \ \ \text{for all\ \ }\delta>0,\\
\label{periodic-4}
    \sup_{\delta>0} \int_{|h|\geq 1} a_{\delta}(h) \nu(h)\,\dup{h} <\infty.
\end{gather}

For each $\delta>0$ we consider the following quadratic form on $L^2(\rd)$ which is defined for Lipschitz continuous
functions with compact support $u,v \in C_0^{\lip}(\rd)$
\begin{equation} \label{form-delta}
    \form^{\delta}(u,v)
    := \iint\limits_{\rd\times\rd} \big(u(x)-u(y)\big)\big(u(x)-u(y)\big) a_{\delta}(y-x) \nu(y-x)\,\dup{y}\,\dup{x}.
\end{equation}
From the assumptions \eqref{periodic-1} and \eqref{periodic-3}, we easily see that $(\form^\delta, C_0^{\lip}(\rd))$
is a closable symmetric form in $L^2(\rd)$ which is translation invariant, see \cite{FOT11}.  Its closure $(\form^{\delta},
\dom^{\delta})$ is a translation invariant regular symmetric Dirichlet form in $L^2(\rd)$, and the associated stochastic
process is a symmetric L\'evy process. If we use \eqref{lkf} and some elementary Fourier analysis, we obtain the following
characterization of the Dirichlet form $(\form^{\delta}, \dom^{\delta})$ based on the characteristic exponent
$\psi_{\delta}$, cf.\ \cite[Example~4.7.28]{Jac01} and \cite[Example~1.4.1]{FOT11},
\begin{gather*}
\left\{\begin{aligned}
    \form^{\delta}(u,v) &= \int_{\rd} \widehat{u}(\xi)\overline{\widehat{v}(\xi)} \psi_\delta(\xi)\, \dup{\xi} \\
    \dom^\delta &= \left\{ u \in L^2(\rd) : \  \int_{\rd} \big|\widehat{u}(\xi)\big|^2 \psi_\delta(\xi)\, \dup{\xi} <\infty \right\},
\end{aligned}\right.
\end{gather*}
$\widehat u(\xi) = (2\pi)^{-d} \int_{\rd} e^{-i\scalp\xi x} u(x)\,\dup{x}$ denotes the Fourier transform and
\begin{equation}\label{symm-levy-exponent}
    \psi_\delta(\xi)=\int_{h\neq 0} \big(1-\cos \scalp\xi h \big) a_\delta(h) \nu(h)\, \dup{h}, \quad \xi\in \rd.
\end{equation}

Condition \eqref{periodic-3} ensures that $a_\delta(h)\nu(h)$ is the density of a L\'evy measure.
If $\nu(h)$ is the density of a L\'evy measure and if $a$ is a bounded, nonnegative (and $1$-periodic) function,
then \eqref{periodic-3} clearly holds. The following example illustrates that for \emph{unbounded} functions $a$
the situation is different.

\begin{example}\label{exa}%
\textbf{a)}\ \
    Let $0<\beta<2$ and pick some $\delta$ such that $0< \delta<1 \wedge(2-\beta)$.
    Define functions $\alpha_0$ on $[0,1/2]$ and $\alpha_1$ on $[0, 1]$ by
    \begin{gather*}
    \alpha_0(x)
    := \begin{cases}
        0, &  x=0,\\
        x^{-\delta}, & 0<x\le \frac 14,\\
        4^{\delta}, & \frac 14 \le x \le \frac 12,
    \end{cases}
    \quad\text{and}\quad
    \alpha_1(x)
    := \begin{cases}
        \alpha_0(x), & 0\le x\le \frac 12,\\
        \alpha_0(1-x), & \frac 12 \le x \le 1.
    \end{cases}
    \end{gather*}
    Denote by $a:\real\to\real$ the $1$-periodic extension of $\alpha_1$ to the real line. It is obvious that
    $a \in L^p_{\loc}(\real)$ for all $1<p<1/\delta$. Define a further function $b=b(x)$ on $\real$ by
    $b(x):=a(x-1/2)$ for $x\in \real$ and set
    \begin{gather*}
        \nu(h)= \frac {b(h)}{|h|^{1+\beta}}, \quad h\neq 0.
    \end{gather*}
    Clearly, $\nu(h)=\nu(-h)$; let us show that $a(h)\nu(h)$ is the density of a L\'evy measure,
    \emph{i.e.}\ $\int_{h\neq 0} \big(1\wedge h^2\big) a(h)\nu(h)\,\dup{h} <\infty$.

    Since $a$ and $\nu$ are even functions, we see
    \begin{align*}
    \int_{h\neq 0} \big(1\wedge h^2\big) a(h) \nu(h)\,\dup{h}
        &= 2 \int_0^1 h^2  a(h) \nu(h)\,\dup{h} + 2 \sum_{\ell=1}^{\infty} \int_{\ell}^{\ell+1} a(h) \nu(h)\,\dup{h}.
    \end{align*}
    For the first term we get
    \begin{align*}
        \int_0^1 h^2  a(h) \nu(h)\,\dup{h}
        &= \int_0^1 h^2 a(h) b(h)  h^{-1-\beta}\,\dup{h}\\
        &= 4^{\delta}\int_0^{1/4} h^{1-\delta-\beta}\,\dup{h}
            + 4^{\delta} \int_{1/4}^{1/2} h^{1-\beta} (1/2-h)^{-\delta}\,\dup{h}\\
        &\quad\mbox{}+ 4^{\delta} \int_{1/2}^{3/4} h^{1-\beta}(h-1/2)^{-\delta}\,\dup{h}
            + 4^{\delta}\int_{3/4}^1 h^{1-\beta} (1-h)^{-\delta}\,\dup{h}\\
        &=: c(\delta)<\infty.
    \end{align*}
    The integrals under the sum appearing in the second term can be estimated using the periodicity of $a$ and $b$; for all $\ell\geq 1$ we have
    \begin{align*}
        \int_{\ell}^{\ell+1} a(h) \nu(h)\,\dup{h}
        &= \int_0^1 a(h+\ell) b(h+\ell) (h+\ell)^{-1-\beta}\,\dup{h}\\
        &= \int_0^1 a(h) b(h) (h+\ell)^{-1-\beta}\,\dup{h}\\
        &\leq \ell^{-1-\beta}\int_0^1 a(h) b(h) \,\dup{h}.
    \end{align*}
   As in the previous calculation, and noting that $0<\delta<1$, we again see that
    \begin{align*}
        \int_0^1 a(h) b(h) \,\dup{h}
        &= 4^{\delta} \int_0^{1/4} h^{-\delta}\,\dup{h} + 4^{\delta} \int_{1/4}^{1/2}  \big(1/2-h)^{-\delta}\,\dup{h} \\
        &\qquad\mbox{} + 4^{\delta} \int_{1/2}^{3/4} (h-1/2)^{-\delta}\,\dup{h} + 4^{\delta} \int_{3/4}^1 (1-h)^{-\delta}\,\dup{h}
        <\infty.
    \end{align*}
   Thus, $c:= \int_0^1 a(h) b(h) \,\dup{h}<\infty$, and
    \begin{gather*}
        \int_{h\neq 0} \big(1\wedge h^2\big) a(h) \nu(h)\,\dup{h}
        \leq 2c(\delta) + c\sum_{\ell=1}^{\infty} \ell^{-1-\beta}
        <\infty.
    \end{gather*}

    On the other hand, we also find that
    \begin{align*}
        \int_{h\neq 0} \big(1\wedge h^2\big) a_{1/2}(h)\nu(h)\,\dup{h}
        &=\int_{h\neq 0} \big(1\wedge h^2\big) a(2h) b(h) |h|^{-1-\beta}\,\dup{h} \\
        &\geq \int_{3/8}^{1/2} h^2 a(2h) b(h) h^{-1-\beta}\,\dup{h}\\
        &= \int_{3/8}^{1/2} h^{1-\beta} (1-2h)^{-\delta} (1/2-h)^{-\delta} \, \dup{h} \\
        &= 2^\delta \int_{3/8}^{1/2} h^{1-\beta} (1-2h)^{-2\delta}\,\dup{h},
    \end{align*}
    and this integral blows up if $0<\beta<3/2$ and $1/2\le \delta<1\wedge(2-\beta)$. In a similar way we can show that
    \begin{gather*}
        \int_{h\neq 0} \big(1\wedge h^2\big) a_{\delta}(h)\nu(h)\,\dup{h} = \infty
    \end{gather*}
    for infinitely many $\delta>0$.

\medskip\noindent
\textbf{b)}\ \
        Let $a=a(x)$ on $\real$ be as in the previous part. Set $\nu(h)=|h|^{-1-\beta}$ for $h\neq 0$.
        Then we can show that this pair $(a, \nu)$ satisfies the conditions \eqref{periodic-1}--\eqref{periodic-3}.
\end{example}

We will now discuss the limit of $(\form^\delta,\dom^\delta)$ as $\delta\downarrow 0$. To this end, we take a sequence of positive numbers $\{\delta_n\}_{n\in\mathds{N}}$ such that $\delta_n \downarrow 0$ as $n\to \infty$. The following result is a standard result from homogenization theory. Usually it is stated in terms of $L^p$ convergence (rather than $L^p_\loc$ convergence), see e.g.\ \cite[Theorem 2.6]{cio-don}.

\begin{lemma}\label{cor-1}
    Suppose that \eqref{periodic-1} and \eqref{periodic-3} hold. The family $\{a_{\delta_n}\}_{n\in\mathbb{N}}$
    converges to the constant $\overline{a} := \int_Q a(h)\,\dup{h}$ weakly in $L^p_{\loc}(\rd)$, $1<p<\infty$,
    i.e.\ for any compact set $K$ of $\rd$,
    \begin{equation}\label{eq:cor-1}
        \lim_{n\to \infty} \int_K g(x) a_{\delta_n}(x)\, \dup{x} = \overline{a} \int_K g(x)\,\dup{x}, \quad g\in L^q(K),
    \end{equation}
    where $p$ and $q$ are conjugate $1/p+1/q=1$.
\end{lemma}

We will need the following corollary of Lemma~\ref{cor-1}.
\begin{corollary}\label{cor-2}
    Assume that \eqref{periodic-1}--\eqref{periodic-3} hold and let $\{\delta_n\}_{n\in\nat}$ be a monotonically decreasing
    sequence of positive numbers such that $\delta_n \to 0$ as $n\to \infty$. For any compact set $K \subset \rd\times\rd$,
    let $g_n\in L^q(K)$ be a sequence of functions which converges in $L^q$ to some $g\in L^q(K)$.
    Then the following limit exists
    \begin{equation}\label{lem-limit-3}
        \lim_{n\to\infty} \iint_{K} g_n(x,y) a_{\delta_n}(x-y)\, \dup{x}\,\dup{y}
        = \overline{a} \iint_{K} g(x,y)\,\dup{x}\,\dup{y}.
    \end{equation}
\end{corollary}
\begin{proof}
Note that
\begin{small}
    \begin{align*}
    &\bigg|{\iint\limits_{K}} g_n(x,y) a_{\delta_n}(x-y)\,\dup{x}\,\dup{y} - \overline{a} {\iint\limits_{K}} g(x,y) \,\dup{x}\,\dup{y} \bigg| \\
    &\leq \bigg|{\iint\limits_{K}} \big(g_n(x,y)-g(x,y)\big) a_{\delta_n}(x-y)\,\dup{x}\,\dup{y}\bigg|
    + \bigg|{\iint\limits_{K}} g(x,y) \big( a_{\delta_n}(x-y)-\overline{a}\big)\,\dup{x}\,\dup{y} \bigg|  \\
    &\leq \bigg[{\iint\limits_{K}} \big|g_n(x,y)-g(x,y)\big|^q\,\dup{x}\,\dup{y} \bigg]^{\frac 1q}
    \bigg[{\iint\limits_{K}} a_{\delta_n}(x{-}y)^p\,\dup{x}\,\dup{y}\bigg]^{\frac 1p}
    + \bigg|{\int\limits_{\rd}} H(z) \big( a_{\delta_n}(z)-\overline{a}\big)\,\dup{z} \bigg|
   \end{align*}
\end{small}
    where we use
    \begin{gather*}
        H(z) := \int_{\rd} \I_{{K}}(y+z,y) g(y+z,y)\,\dup{y},\quad  z\in \rd.
    \end{gather*}
    Since $K$ is a compact subset of $\rd \times \rd$,  $H$ has compact support, hence $H \in L^q_\loc(\rd)$. Because of Lemma~\ref{cor-1}, the second term tends to $0$; the first term also tends to $0$ since $g_n\to g$ in $L^q$, and
    $\sup_{0<\delta<1}\iint_{K} a_{\delta}(x-y)^p\,\dup{x}\,\dup{y}$ is finite. We prove this only for $d=1$, the arguments for $d>1$
    just have heavier notation. Without loss of generality we may assume that $K = L\times L$ for $L=[-N,N]\subset\real$ and $N\in\nat$. Now take $k := \entier{2N/{\delta}}+1 \in\mathds{N}$, the smallest integer which is bigger or equal $2N/{\delta}$. We have
    \begin{align*}
        \iint_{K} a_{\delta}(x-y)^p\,\dup{x}\,\dup{y}
        &= \int_{-N}^N \int_{-N}^N a_{\delta}(x-y)^p\,\dup{x}\,\dup{y}\\
        &=\int_{-N}^N \left(\int_{-N+y}^{N+y} a_{\delta}(x)^p\,\dup{x} \right) \dup{y} \\
        &\leq \int_{-N}^N \left(\int_{-2N}^{2N} a_{\delta} (x)^p\,\dup{x}\right) \dup{y} \\
        &= 2N \int_{-2N}^{2N} a_{\delta} (x)^p \,\dup{x}
        =: 2N\cdot\mathrm{I}
    \end{align*}
    where
    \begin{align*}
        \mathrm{I}
        =\int_{-2N}^{2N} a_{\delta}(z)^p\, \dup{z}
        &= \delta \int_{-2N/\delta}^{2N/\delta} a(z)^p \,\dup{z}\\
        &\leq \delta \int_{-k}^{k-1} a(z)^p \,\dup{z}\\
        &= \delta \sum_{\ell =-k}^{k-1} \int_{\ell}^{\ell+1} a(z)^p \,\dup{z}.
    \end{align*}
    Because of the periodicity of $a$, we find that
    \begin{gather*}
        \mathrm{I} \leq \delta \sum_{\ell =-k}^{k-1} \int_{0}^{1} a(z+\ell)^p \,\dup{z}
        = 2k\delta \int_0^1 a(z)^p\,\dup{z}
        \leq 2(2N + 1)\int_0^1 a(z)^p\,\dup{z}.
    \qedhere
    \end{gather*}
\end{proof}

Recall that a sequence of closed forms  $\{(\form^n,\dom^n)\}_{n\in\nat}$ defined on $L^2(\rd)$ is called \emph{Mosco-convergent}
to a form $(\form,\dom)$,  if the following two conditions are satisfied.  As usual, we extend $\form^n$ and $\form$ to the whole space $L^2(\rd)$ by setting $\form^n(u,u)=\infty$, resp.\ $\form(u,u)=\infty$, if $u\notin\dom^n$, resp.\ $u\notin\dom$.
\begin{description}
\item[\textbf{(M1)}]
    For all $u\in L^2(\rd)$ and all sequences $(u_n)_{n\in\nat}\subset L^2(\rd)$ such that $u_n \rightharpoonup u$ (weak convergence in $L^2$) we have
    $\displaystyle
        \liminf_{n\to\infty} \form^n(u_n,u_n) \geq \form(u,u).
    $
\item[\textbf{(M2)}]
    For every $u\in\dom$ there exist elements $u_n\in\dom^n$, $n\in\nat$, such that $u_n\to u$ (strong convergence in $L^2$) and
    $\displaystyle
        \limsup_{n\to\infty} \form^n(u_n,u_n)\leq\form(u,u).
    $
\end{description}
Note that \textbf{(M1)} entails that we have $\lim_{n\to\infty} \form^n(u_n,u_n)=\form(u,u)$ in \textbf{(M2)}.

\medskip
We can now state the main result of our paper.
Together with Remark~\ref{rem}, it can be seen as the Dirichlet form apporach to the problem discussed in \cite{HIT77} and \cite{T92,Sa16}. The paper \cite{KZ18} has, using completely different techniques, similar results for stable-like operators and forms, which include also some non-symmetric and non-translation invariant settings.
\begin{theorem}\label{thm}
    Assume that \eqref{periodic-1}--\eqref{periodic-4} hold for the functions $a$ and $\nu$,
    and let $\nu$ be locally bounded as a function defined on $\rd\setminus\{0\}$.
    Let $\{\delta_n\}_{n\in\nat}$ be a monotonically decreasing sequence of positive numbers such that
    $\delta_n \to 0$ as $n\to \infty$. For each $n\in\nat$ we consider the Dirichlet forms
    $(\form^n, \dom^n):=(\form^{\delta_n}, \dom^{\delta_n})$ defined in \eqref{form-delta}.
    The Dirichlet forms $(\form^n, \dom^n)$ converge to $(\form, \dom)$ in the sense of Mosco.
    The limit $(\form, \dom)$ is the closure of $(\form, C_0^{\lip}(\rd))$ which is given by
    \begin{gather*}
        \form(u,v)
        := \overline{a} \iint_{\real\times\real}\big(u(x)-u(y)\big)\big(v(x)-v(y)\big) \nu(y-x)\,\dup{y}\,\dup{x}.
    \end{gather*}
   \end{theorem}
\begin{proof}
    We will check the conditions \textbf{(M1)} and \textbf{(M2)} of Mosco convergence. For \textbf{(M1)} we take any $u \in L^2(\rd)$ and any sequence $\{u_n\}\subset L^2(\rd)$ such that $u_n\rightharpoonup u$ as $n\to\infty$. Without loss, we may assume that $\liminf_{n\to\infty} \form^n(u_n,u_n)<\infty$.

    We will use the Friedrichs mollifier. This is a family of convolution operators
    \begin{gather*}
        J_{\epsilon}[u](x)=\int_{\rd} u(x-y)\rho_{\epsilon}(y)\,\dup{y}, \quad  x\in \rd,\; \epsilon>0,
    \end{gather*}
    given by the kernels $\{\rho_{\epsilon}\}_{\epsilon>0}$ for a $C^\infty$-kernel $\rho:\rd\to [0,\infty)$ satisfying
    \begin{gather*}
        0\leq \rho(x)=\rho(-x),\quad \int_{\rd} \rho(x)\,\dup{x}=1, \quad
        \supp \, [ \rho ] = \big\{ x \in \rd : \ |x|\leq 1\big\} \\
 {\rm and} \quad  \rho_{\epsilon}(x):=\rho(x/\epsilon), \quad {\rm for} \ \epsilon>0 \  \ {\rm and} \ \  x\in \rd. \qquad
    \end{gather*}
    We then have
    \begin{align*}
    &\form^n(u_n,u_n)\\
        &\quad= \iint_{x\neq y} \bigl(u_n(x)-u_n(y)\bigr)^2 a_{\delta_n}(y-x) \nu(y-x)\,\dup{y}\,\dup{x} \\
        &\quad= \int_\rd \left(\iint_{x\neq y} \bigl(u_n(x)-u_n(y)\bigr)^2 a_{\delta_n}(y-x) \nu(y-x)\,\dup{y}\, \dup{x}\right)    \rho_{\epsilon}(z)\,\dup{z} \\
        &\quad= \int_\rd \left(\iint_{x\neq y} \bigl(u_n(x-z)-u_n(y-z)\bigr)^2 a_{\delta_n}(y-x) \nu(y-x)\,\dup{y}\, \dup{x}\right) \rho_{\epsilon}(z)\,\dup{z},
    \intertext{and using the Fubini theorem and Jensen's inequality yields, for any compact set $K$ so that $K \subset \rd\times \rd \setminus \{(x,x): x\in\rd \}$,}
    &\form^n(u_n,u_n)\\
        &\quad= \iint_{x\neq y} \left( \int_\rd \bigl(u_n(x-z)-u_n(y-z)\bigr)^2  \rho_{\epsilon}(z)\,\dup{z} \right) a_{\delta_n}(y-x) \nu(y-x)\,\dup{y}\, \dup{x}\\
        &\quad\geq \iint_{x\neq y} \left( \int_{\rd} \bigl(u_n(x-z)-u_n(y-z)\bigr)\rho_{\epsilon}(z)\,\dup{z} \right)^2 a_{\delta_n}(y-x) \nu(y-x)\,\dup{y} \,\dup{x} \\
        &\quad\geq \iint_{K} \bigl(J_{\epsilon}[u_n](x) - J_{\epsilon}[u_n](y) \bigr)^2  a_{\delta_n}(y-x) \nu(y-x)\,\dup{y}\, \dup{x}.
    \end{align*}
    Note that $\sup_{n\in\nat}\|u_n\|_{L^2}<\infty$ because of the weak convergence $u_n\rightharpoonup u$.  Using again weak convergence $u_n\rightharpoonup u$, we conclude that $u_{n,\epsilon} = J_{\epsilon}[u_n]$ converges pointwise to $u_{\epsilon}:=J_{\epsilon}[u]$. Using the local boundedness of $\nu$ on $\rd\setminus\{0\}$ and the fact that $K$ is a compact set satisfying $K \subset \rd \times\rd \setminus\{(x,x): x \in \rd\}$, we see that $\bigl(u_{n,\epsilon}(x) - u_{n,\epsilon}(y) \bigr)^2\nu(y-x)$ converges in $L^q(K)$ to
    \begin{gather*}
        \bigl(u_{\epsilon}(x) - u_{\epsilon}(y) \bigr)^2\nu(y-x)
        \quad \text{as $n\to\infty$}.
    \end{gather*}
    From \eqref{lem-limit-3} we get
    \begin{align*}
    \liminf_{n\to\infty}\form^n(u_n,u_n)
        &\geq \liminf_{n\to\infty}\form^n(u_{n,\epsilon},u_{n,\epsilon}) \\
        &\geq \liminf_{n\to\infty} \iint_{K} \bigl(u_{n,\epsilon}(x) - u_{n,\epsilon}(y) \bigr)^2  a_{\delta_n}(y-x) \nu(y-x)\,\dup{y}\, \dup{x}  \\
        &= \overline{a}\iint_{K} \bigl(u_{\epsilon}(x) - u_{\epsilon}(y) \bigr)^2 \nu(y-x)\,\dup{y}\, \dup{x}.
    \end{align*}
    Since $K \subset \rd\times\rd\setminus\{(x,x) : x\in \rd\}$ is an arbitrary compact set, we can approximate $\rd \times \rd \setminus\{(x,x) : x\in\rd\}$ by such sets. Using monotone convergence and the fact that the left hand side is independent of $K$, we arrive at
    \begin{align} \notag
        \sup_{0<\epsilon<1} \form(u_{\epsilon},u_{\epsilon})
        &= \sup_{0<\epsilon<1} \!\!\!\!\sup_{\substack{K \text{\ compact}\\ K \subset \rd \times \rd \setminus\{(x,x): x\in \rd\}}}  \!\!\!\!
        \overline{a}\iint_{K} \bigl(u_{\epsilon}(x) - u_{\epsilon}(y) \bigr)^2 \nu(y-x)\,\dup{y}\, \dup{x}\\
        \label{limit0}
        &\leq \liminf_{n\to\infty}\form^n(u_n,u_n) < \infty.
    \end{align}%
    Theorem~2.4 in \cite{SU12} now shows that $u_{\epsilon} \in \dom\cap C^{\infty}(\rd)$ for each $\epsilon\in (0,1)$. Since $J_\epsilon$ is an $L^2$-contraction operator for each $\epsilon>0$, we see that the family $\{u_\epsilon\}_{\epsilon>0}$, $u_\epsilon = J_{\epsilon}[u]$, is bounded w.r.t.\ $\form_1(\sbullet,\sbullet):=\form(\sbullet,\sbullet)+(\sbullet,\sbullet)_{L^2}$ by \eqref{limit0}. The Banach--Alaoglu theorem guarantees that there is an $\form_1$-weakly convergent subsequence $u_{\epsilon(n)}$, $\epsilon(n)\downarrow 0$, and a function $v$ so that $u_{\epsilon(n)}$ converges $\form_1$-weakly to $v\in\dom$. Using the Banach--Saks theorem shows that the Ces\`aro means $\frac 1n \sum_{k=1}^n u_{\epsilon(n_k)}$ of a further subsequence converge $\form_1$-strongly, hence in $L^2(\rd)$, to $v$. As $u_\epsilon$ converges to $u$ in $L^2(\rd)$, we can identify the limit as $u=v$. In particular, $u \in \dom$ and
    \begin{equation*}
        \liminf_{n\to \infty} \form^{n}(u_n,u_n)\geq \form (u,u).
    \end{equation*}
    In order to see  \textbf{(M2)}, we use the regularity of the Dirichlet form $(\form, \dom)$; therefore, it is enough to consider $u \in C_0^{\lip}(\rd)$. Set $u_n=u\in C_0^{\lip}(\rd)$ for each $n$, and $L := \supp u$ and $G:= L+B_1(0)$. Because of the symmetry of the form we have
    \begin{equation*}
        \form^n(u,u) = \form^n_{G\times G}(u,u)+ 2\form^n_{G\times G^c}(u,u)
    \end{equation*}
    where, using the fact that $L=\supp u\subset G$,
    \begin{align*}
        \form^n_{G\times G}(u,u)
        &= \iint_{G\times G} \big(u(x)-u(y)\big)^2 a_{\delta_n}(y-x) \nu(y-x)\,\dup{y}\,\dup{x},\\
        \form^n_{G\times G^c}(u,u)
        &= \iint_{L\times G^c} u^2(x) a_{\delta_n}(y-x) \nu(y-x)\,\dup{y}\,\dup{x}.
    \end{align*}
    Using Corollary~\ref{cor-2} we see that
    \begin{equation*}
        \lim_{n\to\infty} \form^n_{G\times G}(u,u) = \overline{a} \iint_{G\times G} \big(u(x)-u(y)\big)^2 \nu(y-x)\,\dup{y}\,\dup{x}.
    \end{equation*}
    For the other part we get
    \begin{align*}
        \form^n_{G\times G^c}(u,u)
        &= \int_\rd \left[\int_L u^2(x)\I_{G^c}(x+h)\,\dup{x}\right] \nu(h) a_{\delta_n}(h)\,\dup{h}\\
        &\leq \epsilon + \int_{|h|\leq R} \left[\int_L u^2(x)\I_{G^c}(x+h)\,\dup{x}\right] \nu(h) a_{\delta_n}(h)\,\dup{h}
    \end{align*}
    for any $\epsilon>0$ and some suitable $R=R_\epsilon$; note that $\epsilon$ and $R$ can be chosen independently of $n$. This is due to our assumption \eqref{periodic-4} and the fact that the expression in the square brackets is a continuous bounded function in $h$. Now we can use Lemma~\ref{cor-1} for the limit $n\to\infty$; if we then let $R\to\infty$ and $\epsilon\to 0$, we get
    \begin{align*}
        \limsup_{n\to\infty}\form^n_{G \times G^c}(u,u)
        &\leq \overline{a} \iint_{\rd\times L} u^2(x)\I_{G^c}(x+h)\nu(h)\,\dup{x} \,\dup{h}\\
        &= \overline{a} \iint_{L\times G^c} u^2(x)\,\nu(x-y)\,\dup{y}\,\dup{x}.
    \end{align*}
    Combining all of the above calculations, it follows that
    \begin{align*}
    \limsup_{n\to \infty}  \form^n(u_n, u_n)
    &= \limsup_{n\to \infty} \Big( \form^n_{G\times G} (u,u) + 2 \form^n_{G\times G^c} (u,u) \Big)  \\
    &\leq \overline{a} \iint_{G\times G} \big(u(x)-u(y)\big)^2 \nu(y-x)\,\dup{y}\,\dup{x} \\
    &\qquad \mbox{} + 2 \overline{a} \iint_{L\times G^c} \big(u(x)-u(y)\big)^2 \nu(y-x)\,\dup{y}\,\dup{x}  \\
    & = \form(u,u),
    \end{align*}
    finishing the proof.
\end{proof}

\begin{remark}\label{rem}
    Suppose that the function $a$ on $\real$ satisfies \eqref{periodic-1}--\eqref{periodic-3}, and $\nu$ is given by $\nu(x)=|x|^{-1-\alpha}$, $x\in \real\setminus\{0\}$, for some $0<\alpha<2$.  Then the following quadratic form defines a translation invariant regular symmetric Dirichlet form on $L^2(\real)$:
    \begin{gather*}
        \tilde{\form}(u,v)
        :=\iint\limits_{x\neq y} \big(u(x)-u(y)\big)\big(u(x)-u(y)\big) \frac{a(x-y)}{|x-y|^{1+\alpha}} \,\dup{x} \,\dup{y},
        \quad u,v \in C_0^{\lip}(\real).
    \end{gather*}
    Let $\tilde X = (\tilde{X}(t))_{t\geq 0}$ be the symmetric L\'evy process on $\real$ associated with the Dirichlet form $(\tilde{\form}, \tilde{\dom})$ on $L^2(\real)$. For any $n\in\mathds{N}$, set
    \begin{gather*}
        X^{(n)}(t):= \epsilon_n \tilde{X}(\epsilon_n^{-\alpha}t), \quad t>0.
    \end{gather*}
    Then $X^{(n)} = (X^{(n)}(t))_{t\geq 0}$ is also a symmetric L\'evy process
    and we denote for each $n\in\nat$ by $(\form^{(n)}, \dom^{(n)})$ the corresponding Dirichlet form. The semigroup $\{T^{(n)}_t\}_{t>0}$ generated by $(\form^{(n)}, \dom^{(n)})$ is given by
    \begin{align*}
        T^{(n)}_t f(x)
        &=\mathds{E}\left[f(X^{(n)}(t)) \:\middle\vert\: X^{(n)}(0)=x\right]\\
        &=\mathds{E}_{x/\epsilon_n}\left[f(\epsilon_n \tilde{X}(\epsilon_n^{-\alpha}t))\right]
        =\left(\tilde{T}_{\epsilon_n^{-\alpha}t}f(\epsilon_n \cdot)\right)(\epsilon_n^{-1}x),
        \quad  x\in \real.
    \end{align*}
    Since the Dirichlet form $(\form^{(n)},\dom^{(n)})$ can be obtained by
    \begin{gather*}
        \form^{(n)}(u,v)=\lim_{t \downarrow 0} \frac 1t \big(u-T^{(n)}_tu,v\big)_{L^2},
    \end{gather*}
    it follows for $t>0$ that
    \begin{align*}
        \frac 1t \big(u-T^{(n)}_tu,v\big)_{L^2}
        &= \frac 1t \int_\real \big[ u(x) -T^{(n)}_t u(x)\big] v(x)\,\dup{x} \\
        &= \frac 1t \int_\real \big[ u\big(\epsilon_n \cdot \epsilon^{-1}_n x\big) - \big(\tilde{T}_{\epsilon_n^{-\alpha}t} u(\epsilon_n \sbullet)\big) \big(\epsilon^{-1}_n x\big) \big] v(x)\,\dup{x} \\
        &= \frac 1{\epsilon_n^{\alpha}} \cdot \frac 1s \int_\real \big[ u(\epsilon_n \xi) - \big(\tilde{T}_s u(\epsilon_n \sbullet)\big)(\xi)\big] v(\epsilon_n \xi) \epsilon_n \, \dup{\xi}
    \end{align*}
    where we use the notation $\xi=\epsilon_n^{-1}x$ and $s=\epsilon_n^{-\alpha} t$. Letting $s\to 0$, hence $t\to 0$, yields
    \begin{align*}
        \lim_{t\to 0}&\frac 1t \big(u-T^{(n)}_tu,v\big)_{L^2}\\
        &= \epsilon_n^{1-\alpha} \cdot \tilde{\form} \big(u(\epsilon_n \sbullet), v(\epsilon_n \sbullet) \big) \\
        &= \epsilon_n^{1-\alpha} \iint_{x\neq y} \big(u(\epsilon_n x)-u(\epsilon_n y)\big)\big(v(\epsilon_n x)-v(\epsilon_n y)\big) \frac{a(x-y)}{|x-y|^{1+\alpha}}\,\dup{x}\,\dup{y} \\
        &= \epsilon_n^{1-\alpha} \iint_{x\neq y} \big(u(x)-u(y)\big)\big(v(x)-v(y)\big) \frac{a\big(\epsilon_n^{-1}(x-y)\big)}{|x-y|^{1+\alpha}} \epsilon_n^{1+\alpha} \frac{\dup{x}}{\epsilon_n}\frac{\dup{y}}{\epsilon_n}  \\
        &= \iint_{x\neq y} \big(u(x)-u(y)\big)\big(v(x)-v(y)\big) \frac{a\big(\epsilon_n^{-1}(x-y)\big)}{|x-y|^{1+\alpha}} \,\dup{x}\,\dup{y}\\
        &=\form^{(n)}(u,v).
    \end{align*}
    Since Mosco convergence entails the convergence of the semigroups, hence the finite-dimensional distributions (fdd) of the processes,
    we may combine the above calculation with Theorem~\ref{thm} to get the following result: {\itshape The processes $X^{(n)}$ associated
    with $(\form^{(n)}, \dom^{(n)})$ -- these are obtained by scaling $t\mapsto \epsilon_n^{-\alpha}t$ and $x\mapsto \epsilon_n x$ from the process
    $\tilde X$ given by $(\tilde{\form}, \tilde{\dom})$ -- converge, in the sense of finite-dimensional distributions, to the process $X$ associated with $(\form, \dom)$.}
\end{remark}


\small\itshape\flushright
Ren\'e L.\ Schilling:
Institute of Stochastics,
Faculty of Mathematics,\\
Technische Universit\"{a}t Dresden,
01062 Dresden, Germany. \\
\texttt{rene.schilling@tu-dresden.de}

\medskip
Toshihiro Uemura:
Department of Mathematics,
Faculty of Engineering Science,\\
Kansai University,
Suita, Osaka 564-8680, Japan \\
\texttt{t-uemura@kansai-u.ac.jp}
\end{document}